\documentclass[12pt]{amsart}

\usepackage{geometry}
\usepackage{xcolor}
\usepackage{graphicx}
\usepackage{enumerate}

\usepackage{mathtools}
\usepackage{amssymb}

\usepackage[numbers]{natbib}

\usepackage[pdfpagelabels,
	bookmarksopen,
	pdfstartview={FitV},
	plainpages=false,
	breaklinks=true]
	{hyperref}

\newcommand{\n}{ \ensuremath{ ^{(n)} } }
\newcommand{\Xhat}{ \ensuremath{ \widehat{X} } }
\newcommand{\Xeq}{ \ensuremath{ \Xhat\n } }
\newcommand{\Ximp}{ \ensuremath{ \widehat{X}^{n,\mathrm{imp}} } }
\newcommand{\Xbar}{ \ensuremath{ \overline{X}^{n} } }
\newcommand{\Yhat}{ \ensuremath{ \widehat{Y} } }
\newcommand{\Yeq}{ \ensuremath{ \Yhat\n } }
\newcommand{\Zbar}{ \ensuremath{ \overline{Z}^{n} } }

\newcommand{\adap}{\ensuremath{ \mathrm{ad} }}
\newcommand{\equi}{\ensuremath{ \mathrm{eq} }}

\newcommand{\N}{\mathbb{N}}

\newcommand{\R}{\mathbb{R}}

\newcommand{\F}{\mathcal{F}}
\newcommand{\A}{\mathcal{A}}

\newcommand{\ind}{1}

\newcommand{\abs}[1]{\ensuremath{ \left\vert #1 \right\vert } }
\newcommand{\set}[1]{\ensuremath{ \left\{ #1 \right\} } }

\DeclareMathOperator{\argmax}{argmax}
\DeclareMathOperator{\sgn}{sgn}

\newcommand{\cond}{\,\middle\vert\,}
\newcommand{\E}[1]{\ensuremath{ \mathrm{E}\left( #1 \right) }}
\newcommand{\prob}{\ensuremath{ \mathrm{P} }}

\newcommand{\ndist}[2]{\mathcal{N}(#1, #2)}
\newcommand{\eqdist}{ \ensuremath{ \stackrel{\text{d}}{=} } }
\newcommand{\todist}{ \ensuremath{ \stackrel{\text{d}}{\longrightarrow} } }
\newcommand{\eps}{\varepsilon}

\theoremstyle{plain}
	\newtheorem{lem}{Lemma}
	\newtheorem{thm}{Theorem}
	\newtheorem{cor}{Corollary}
	\newtheorem{prop}{Proposition}
\theoremstyle{definition}

\theoremstyle{remark}
	\newtheorem{rem}{Remark}

\begin{document}

\title[Optimal Strong Approximation of the $1$-D Squared {B}essel Process]
	{Optimal Strong Approximation of the One-dimensional Squared {B}essel Process}

\author[Hefter]{Mario Hefter}
\address{Fachbereich Mathematik\\
Technische Universit\"at Kaisers\-lautern\\
Postfach 3049\\
67653 Kaiserslautern\\
Germany}
\email{hefter@mathematik.uni-kl.de}

\author[Herzwurm]{Andr\'{e} Herzwurm}
\address{Fachbereich Mathematik\\
Technische Universit\"at Kaisers\-lautern\\
Postfach 3049\\
67653 Kaiserslautern\\
Germany}
\email{herzwurm@mathematik.uni-kl.de}

\begin{abstract}
	We consider the one-dimensional squared Bessel process given by the stochastic differential
	equation (SDE)
	\begin{align*}
		dX_t = 1\,dt + 2\sqrt{X_t}\,dW_t, \quad X_0=x_0, \quad t\in[0,1],
	\end{align*}
	and study strong (pathwise) approximation of the solution $X$ at the final time point $t=1$.
	This SDE is a particular instance of a Cox-Ingersoll-Ross (CIR) process where the boundary point
	zero is accessible.
	We consider numerical methods that have access to values of the driving Brownian motion $W$
	at a finite number of time points. We show that the polynomial convergence rate
	of the $n$-th minimal errors for the class of adaptive algorithms as well as
	for the class of algorithms that rely on equidistant grids are equal to
	infinity and $1/2$, respectively.
	This shows that adaption results in a tremendously improved convergence rate.
	As a by-product, we obtain that the parameters appearing in the CIR process affect
	the convergence rate of strong approximation.
\end{abstract}

\keywords{
	Cox-Ingersoll-Ross process;
	strong approximation;
	$n$-th minimal error;
	adaptive algorithm;
	reflected Brownian motion}

\maketitle

\section{Introduction}
In recent years, strong approximation of stochastic differential equations (SDEs) has intensively
been studied for SDEs of the form
\begin{align}\label{eq:CIR}
	dX_t = (a-bX_t)\,dt + \sigma\sqrt{X_t}\,dW_t, \quad X_0=x_0, \quad t\geq0,
\end{align}
with a one-dimensional Brownian motion $W$, and $a,x_0\geq0$, $b\in\R$, and $\sigma>0$.
These SDEs are known to have a unique non-negative strong solution.
Such SDEs were proposed in \cite{cir} as a model for short-term interest rates.
The solution is called Cox-Ingersoll-Ross (CIR) process. Moreover, CIR processes
are used as the volatility process in the Heston model~\cite{heston}.

Strong approximation is of particular interest due to the multi-level Monte Carlo
technique, see \cite{giles1,giles2,heinrich1}. For an optimality result of
this technique applied to quadrature problems of the form $\E{f(X)}$
with $f\colon C([0,1])\to\R$, we refer to \cite{creutzig}. In mathematical finance,
the functional $f$ often represents a discounted payoff of some derivative
and $\E{f(X)}$ is the corresponding price.

In \cite{alfonsi2005}, various numerical schemes have been proposed and numerically
tested for the SDE~\eqref{eq:CIR} with different choices of the corresponding parameters.
These numerical results indicate a convergence at a polynomial rate,
which depends on the parameters $a,\sigma$. More precisely,
the empirical convergence rate is monotonically decreasing in the quotient $\sigma^2/(2a)$
for all numerical schemes that have been tested.
Polynomial convergence rates for strong approximation of~\eqref{eq:CIR}
have been proven by \cite{berkaoui,dereich,alfonsi2013,neuenkirch-szpruch,hjn-cir},
where either a global or the final time error w.r.t.~the $L_p$-norm is studied.
All these results only hold for some parameter range within $\sigma^2/(2a)<2$
and share the same monotonicity, see Figure~\ref{fig:CIR-rates}.
For an overview of numerical schemes and results on strong convergence without a rate
we refer to \cite{dereich} and the references therein.

\begin{figure}[htp]
	\centering
	\includegraphics[width=0.9\linewidth]{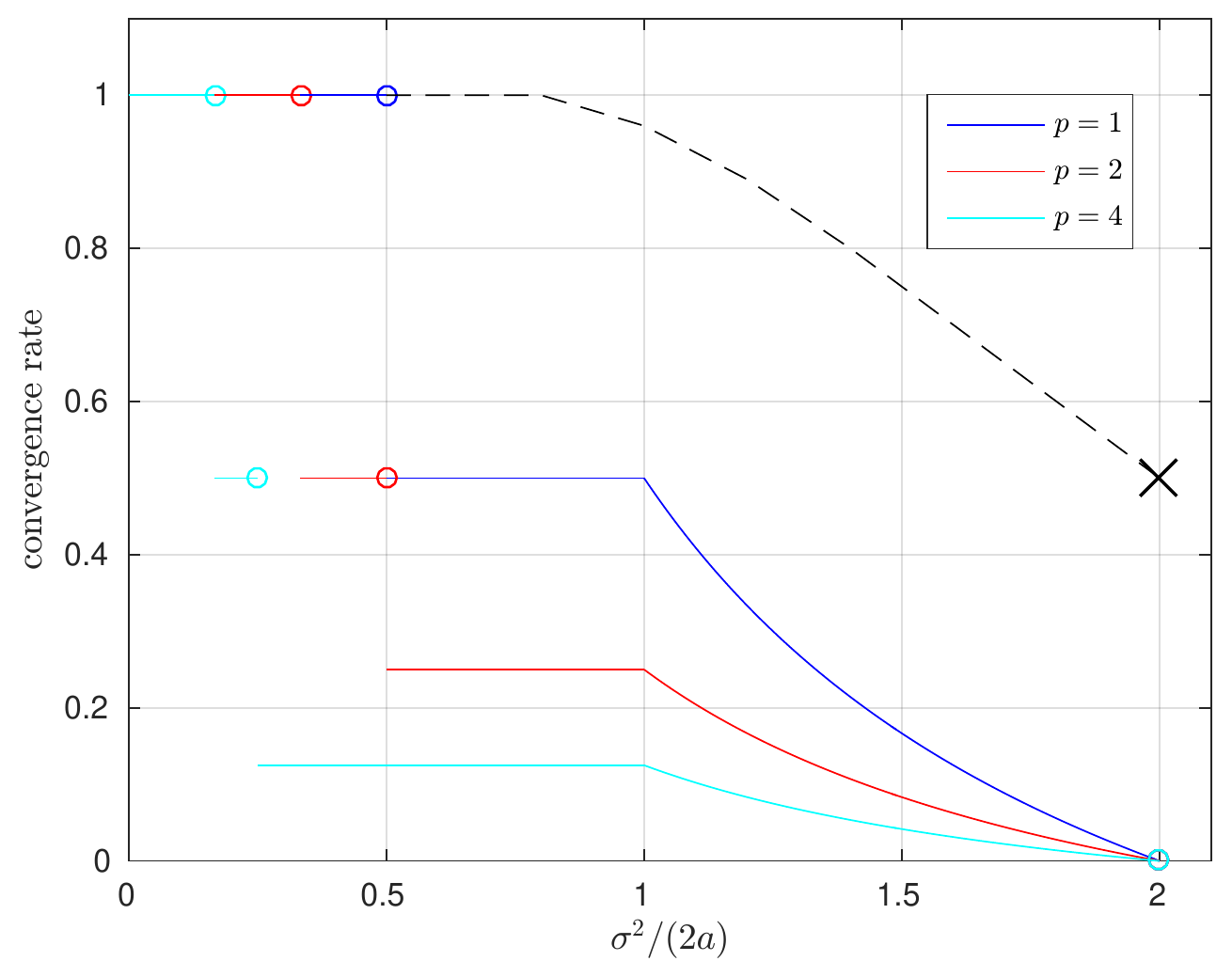}
	\caption{Convergence rates of the known (best) upper bounds from
		\cite{dereich,alfonsi2013,neuenkirch-szpruch,hjn-cir}
		are shown for different values of $p\in[1,\infty[$
		with error given by~\eqref{eq:error-scheme}.
		All these upper bounds are obtained for the drift-implicit Euler scheme.
		The dashed line shows the corresponding empirical convergence
		rates from \cite{alfonsi2005} for $p=1$.
	}
	\label{fig:CIR-rates}
\end{figure}

In this paper we consider the particular case of the SDE~\eqref{eq:CIR} with
\begin{align*}
	\sigma^2/(2a)=2.
\end{align*}
By rescaling with $(2/\sigma)^2$, we thus may restrict ourselves to SDEs of the form
\begin{align}\label{eq:bessel-1-b}
	dX_t &= (1-bX_t)\,dt + 2\sqrt{X_t}\,dW_t, \quad X_0=x_0, \quad t\geq0,
\end{align}
with $x_0\geq0$ and $b\in\R$. Furthermore, we focus on the particular instance
of~\eqref{eq:bessel-1-b} with $b=0$, i.e.,
\begin{align}\label{eq:bessel-1}
	dX_t = 1\,dt + 2\sqrt{X_t}\,dW_t, \quad X_0=x_0, \quad t\geq0.
\end{align}
Its solution is called the square of a $1$-dimensional Bessel process.
For a detailed study of (squared) Bessel processes we refer to \cite[Chap.~XI]{yor}.

\bigskip
In the context of (strong) approximation of SDEs, the majority of numerical methods
in the literature are non-adaptive \cite{kloeden-platen}.
A non-adaptive algorithm uses a fixed discretisation of the driving Brownian motion
whereas adaptive algorithms may sequentially choose the evaluation points.
The most frequently studied methods in the class of non-adaptive algorithms are
Euler or Milstein-type methods that are based on values of the driving Brownian motion
on an equidistant grid. For various strong approximation problems of SDEs
satisfying standard assumptions, adaption does not help up to a multiplicative constant,
see, e.g., \cite{gronbach-ritter}.
In particular, \cite{gronbach2} showed for strong approximation of scalar SDEs at
the final time point that no adaptive method can be better
(up to a multiplicative constant) than the classical Milstein scheme.
Let us stress that these standard assumptions are not fulfilled
by~\eqref{eq:bessel-1} since the diffusion coefficient is not even locally
Lipschitz continuous.

In contrast to that, the main result of this paper is that adaptive methods
are far superior to methods that are based on an equidistant grid for strong approximation
of the solution $X_1$ of~\eqref{eq:bessel-1}.
For this, we determine the polynomial convergence rate of the corresponding
$n$-th minimal errors, which will be introduced below.

\bigskip
Let $X_1$ be the solution of \eqref{eq:bessel-1} at time $t=1$, and let $p\in[1,\infty[$.
The error of an approximation $\Xhat_1$ of $X_1$ is defined by
\begin{align}\label{eq:error-scheme}
	e_p(\Xhat_1) = \left( \E{ \big\vert X_1-\Xhat_1\big\vert^p } \right)^{1/p}.
\end{align}
At first, we consider the class of methods that only use values of the driving
Brownian motion $W$ on an equidistant grid with $n$ points given by
\begin{align*}
	\mathfrak{C}^\equi(n) = \left\{
			\Xhat_1 = \mathrm{\Phi}(W_{\frac 1n},W_{\frac 2n},\ldots,W_1):\
				\mathrm{\Phi}\colon \R^n\to\R\text{ Borel-measurable}
		\right\}.
\end{align*}
The corresponding $n$-th minimal error for the approximation of $X_1$ is given by 
\begin{align}\label{eq:minimal-error-equi}
	e_p^\equi(n) = \inf \left\{ e_p(\Xhat_1):\
		\Xhat_1 \in \mathfrak{C}^\equi(n) \right\}.
\end{align}
Roughly speaking, $e_p^\equi(n)$ is the error of the best algorithm for the
approximation of~$X_1$ w.r.t.~the $L_p$-norm that only uses
$W_{\frac 1n},W_{\frac 2n},\ldots,W_1$.
Clearly, Euler and Milstein-type schemes fit into this class of algorithms.
In the case $p=2$, the optimal approximation is given by the conditional
expectation of $X_1$ given the $\sigma$-algebra generated by
$W_{\frac 1n},W_{\frac 2n},\ldots,W_1$.

The class of adaptive methods that use values of the driving Brownian motion $W$
at $n$ sequentially chosen points is given by
\begin{align*}
	\mathfrak{C}^\adap(n) = \Big\{
		\Xhat_1 &= \mathrm{\Phi}(W_{t_1},\ldots,W_{t_n}):\
			\mathrm{\Phi}\colon\R^n\to\R \text{ Borel-measurable}, \\
		& t_1\in[0,1], \\
		& t_2 = \varphi_2(W_{t_1}),\ \varphi_2\colon\R\to[0,1] \text{ Borel-measurable}, \\
		& \vdots \\
		& t_n = \varphi_n(W_{t_1},\ldots,W_{t_{n-1}}),\
		\varphi_{n}\colon\R^{n-1}\to[0,1] \text{ Borel-measurable}
	\Big\}.
\end{align*}
Here, in contrast to the class $\mathfrak{C}^\equi$, the $k$-th evaluation site $t_k$
may depend on the previous $k-1$ observations of $W$. Moreover, considering the
particular choice of constant mappings $\varphi_k=k/n$ yields
$\mathfrak{C}^\equi(n)\subseteq\mathfrak{C}^\adap(n)$ for all $n\in\N$.
The $n$-th minimal error for the approximation of $X_1$
for the class of adaptive methods is given by
\begin{align}\label{eq:minimal-error-adap}
	e_p^\adap(n) = \inf \left\{ e_p(\Xhat_1):\
		\Xhat_1 \in \mathfrak{C}^\adap(n) \right\}.
\end{align}
We clearly have $e_p^{\adap}(n) \leq e_p^\equi(n)$ for all $n\in\N$.

\bigskip
In the following we present our main results. We write $a_n\preccurlyeq b_n$ for
sequences of non-negative reals $a_n$ and $b_n$ if there exists a constant $c>0$
such that $a_n\leq c\cdot b_n$ for all $n\in\N$. Moreover, we write $a_n\asymp b_n$
if $a_n\preccurlyeq b_n$ and $b_n\preccurlyeq a_n$.

We show that the polynomial convergence rate of the $n$-th minimal error $e_p^\equi$
is equal to $1/2$ for all $p\in{[1,\infty[}$.
More precisely, Corollary~\ref{cor:convergence-equi} yields
\begin{align}\label{eq:asymp}
	e_p^\equi(n) \asymp n^{-1/2}
\end{align}
for all $p\in{[1,\infty[}$. Of course, the constants hidden in the
``$\asymp$''-notation may depend on $p$. Furthermore, the corresponding
upper bound is attained by the drift-implicit Euler scheme $\Ximp_1$,
see Theorem~\ref{thm:upper-bound} and \eqref{eq:drift-implicit-Euler-scheme-2}.
In the more general case of the SDE~\eqref{eq:bessel-1-b},
the drift-implicit Euler scheme is given by
\begin{align}\label{eq:drift-implicit-Euler-scheme}
	\Ximp_\frac{k+1}{n} = \left( \frac{
				\sqrt{ \Ximp_\frac{k}{n} } + \left(W_{\frac{k+1}{n}}-W_{\frac kn}\right)
				+ \sqrt{ \left(\sqrt{ \Ximp_\frac{k}{n} } + \left(W_{\frac{k+1}{n}}-W_{\frac kn}\right) \right)^2 }
			}{ 2+\frac bn }
			\right)^2
\end{align}
for $k=0,\ldots,n-1$ and $\Ximp_0=x_0$. Let us mention that the drift-implicit Euler
scheme is actually proposed for the SDE \eqref{eq:CIR} with parameters satisfying
$\sigma^2/(2a)<2$, see \cite{alfonsi2005}. Nevertheless, it is still well defined
in the limiting case $\sigma^2/(2a)=2$.
Note that the upper bound from \eqref{eq:asymp} is the first strong convergence result
with a positive rate in the case $\sigma^2/(2a)=2$, cf.~Figure~\ref{fig:CIR-rates}.

For adaptive algorithms the situation is rather different.
Corollary~\ref{cor:convergence-adap} shows that
\begin{align}\label{eq:precadap}
	e_p^{\adap}(n) \preccurlyeq n^{-q}
\end{align}
for all $p\in{[1,\infty[}$ and for all $q\in{[1,\infty[}$. Hence the polynomial
convergence rate of the $n$-th minimal error $e_p^\adap$ is equal to infinity.
More precisely, for every $q\in{[1,\infty[}$ we construct an adaptive algorithm
that converges (at least) at a polynomial rate~$q$, see Theorem~\ref{thm:adaptive}.
In fact, numerical experiments suggest an exponential decay, see
Figure~\ref{fig:numerical-results}. Moreover, such algorithms can be easily
implemented on a computer with number of operations of order $n^2$.
Combining \eqref{eq:asymp} and \eqref{eq:precadap} establishes our claim
that adaptive algorithms are far superior to non-adaptive algorithms that are
based on equidistant grids for strong approximation of~\eqref{eq:bessel-1}.
Let us stress that this is the first result on SDEs where adaption
results in an improved convergence rate compared to methods that are
based on equidistant grids.

A key step for the proofs of \eqref{eq:asymp} and \eqref{eq:precadap} consists of
identifying the pathwise solution of \eqref{eq:bessel-1}, see Proposition~\ref{prop:sol-b},
and link this problem to global optimization under the Wiener measure.
Let us mention that the analysis of the adaptive algorithm in Theorem~\ref{thm:adaptive}
heavily relies on results of~\cite{chh2015}.

Although we have shown that adaptive algorithms are far superior to
methods that are based on an equidistant grid for a particular choice
of the parameters of SDE~\eqref{eq:CIR}, it is open whether this
superiority also holds for more general parameter constellations.

\bigskip
We now turn to the more general case of the SDE~\eqref{eq:bessel-1-b} with $b\in\R$.
Moreover, we consider a stronger error criterion which is pathwise given by the
supremum norm. In this case we obtain
\begin{align}\label{eq:upper-bound-b}
	\left( \E{ \sup_{0\leq t\leq 1} \abs{X_t-\Xbar_t}^p } \right)^{1/p}
		\preccurlyeq n^{-1/2}\cdot \sqrt{\ln(1+n)}
\end{align}
for all $p\in{[1,\infty[}$, where $\Xbar$ denotes a projected equidistant
Euler scheme, see Remark~\ref{rem:sup-norm}. This scheme coincides
with the drift-implicit Euler scheme for $b=0$.
Let us stress that this error bound is the first strong
convergence result with a positive rate in the case $\sigma^2/(2a)=2$ and
arbitrary $b\in\R$, cf.~Figure~\ref{fig:CIR-rates}.
At present, we have only shown the upper bound~\eqref{eq:upper-bound-b}.
Nevertheless, we expect this upper bound to be sharp even for
adaptive algorithms.

\bigskip
Let us briefly comment on some consequences of the results presented above for
strong approximation of CIR processes.

It is well-known that the parameters $a$, $b$, and $\sigma$ in SDE~\eqref{eq:CIR}
have an influence on the behavior of its solution. For instance, the solution
remains strictly positive (the boundary point $0$ is inaccessible) if and only if
the so-called Feller condition $\sigma^2/(2a)\leq 1$ is satisfied.
As illustrated in Figure~\ref{fig:CIR-rates}, the drift-implicit Euler scheme converges at
least with rate $1$ if $\sigma^2/(2a)<\min(2/(3p),1/2)$, see \cite{alfonsi2013,neuenkirch-szpruch},
and so does the corresponding $n$-th minimal error for methods using an equidistant grid.
Hence the quotient $\sigma^2/(2a)$ affects the convergence rate
of the $n$-th minimal error for equidistant methods since it drops down
to $1/2$ for $\sigma^2/(2a)=2$ and $b=0$ according to~\eqref{eq:asymp}.

In contrast to the known upper bounds, cf.~Figure~\ref{fig:CIR-rates},
the convergence rate of the drift-implicit Euler scheme for~\eqref{eq:bessel-1}
does not depend on the $L_p$-norm appearing in the error criterion,
see \eqref{eq:asymp}.

Let us comment on lower bounds for strong approximation of SDEs
at the final time point based on the values of the driving Brownian motion.
In \cite{clark-cameron}, a two-dimensional SDE is presented where
the corresponding convergence rate is shown to be $1/2$. In contrast
to rate $1$ for smooth scalar SDEs \cite{gronbach2}, the difficulty
in \cite{clark-cameron} arises from the presence of L\'{e}vy areas.
More recently, the existence of SDEs with smooth coefficients has been shown
where the corresponding $n$-th minimal error converges arbitrarily slow to zero,
see \cite{Hairer2015,yaroslavtseva}.
It is crucial that these SDEs are multi-dimensional.
Apart from \eqref{eq:asymp}, we are not aware of any other lower bound
with convergence rate less than $1$ for a scalar SDE.

\bigskip
This paper is organized as follows. In Section~\ref{sec:bessel} we derive an explicit
representation of the solution of \eqref{eq:bessel-1} and the more general case
of~\eqref{eq:bessel-1-b}.
Using this representation we show sharp upper and lower bounds of $e_p^\equi$
in Section~\ref{sec:equi}. In Section~\ref{sec:adaptive} we consider a particular
adaptive method that achieves an arbitrarily high polynomial convergence rate.
Finally, we illustrate our results by numerical experiments.

\section{Squared {B}essel Process of Dimension One}\label{sec:bessel}
In this section, we will derive an explicit expression for the strong solution
of~\eqref{eq:bessel-1} by using basic results about reflected SDEs.
Subsequently, we will extend this technique to the more general case of
SDE~\eqref{eq:bessel-1-b}.

\bigskip
In the following let $(\Omega,\F,\prob)$ be a complete probability space
and let $(\F_t)_{t\geq0}$ be a filtration on this space
satisfying the usual conditions.

Given $x_0\geq0$ and a Brownian motion $B$ w.r.t.\ $(\F_t)_{t\geq0}$, we define
\begin{align}\label{eq:def-W}
	W_t = \int_0^t \sgn(B_s+\sqrt{x_0})\,dB_s
\end{align}
for all $t\geq0$ with
$\sgn=\ind_{\{x>0\}}-\ind_{\{x\leq0\}}$.
Then, $W$ is a Brownian motion w.r.t.\ $(\F_t)_{t\geq0}$. Indeed,
the quadratic variation of $W$ satisfies
\begin{align*}
	[W]_t = \int_0^t \sgn(B_s+\sqrt{x_0})^2\,ds = \int_0^t 1\,ds = t,
\end{align*}
and thus L\'{e}vy's characterization can be applied. Now, consider the SDE~\eqref{eq:bessel-1}
where the driving Brownian motion $W$ has the particular form given in \eqref{eq:def-W}.
Due to this construction of $W$, we see that the solution
of \eqref{eq:bessel-1} is given by
\begin{align}\label{eq:bessel-1-sol-B}
	X_t=(B_t+\sqrt{x_0})^2,
\end{align}
since $\sqrt{X_t}=\abs{B_t+\sqrt{x_0}}$ and hence
\begin{align*}
	2\int_0^t \sqrt{X_s}\,dW_s
		&= 2\int_0^t \abs{B_t+\sqrt{x_0}}\cdot\sgn(B_s+\sqrt{x_0})\,dB_s \\
	&= 2\int_0^t (B_s+\sqrt{x_0})\,dB_s \\
	&= B_t^2-t + 2\sqrt{x_0}\, B_t.
\end{align*}
Moreover, Tanaka's formula \cite[Prop.\ III.6.8]{karatzas} given by
\begin{align*}
	\abs{B_t-a} = \abs{a} + \int_0^t \sgn(B_s-a)\,dB_s + 2L_t^B(a), \quad a\in\R,
\end{align*}
where $L^B(a)$ denotes the local time of $B$ in $a$, yields for $a=-\sqrt{x_0}$ that
\begin{align*}
	\abs{B_t+\sqrt{x_0}} &= \sqrt{x_0} + W_t + 2L_t^B(-\sqrt{x_0}) \\
	&= \sqrt{x_0} + W_t
		+ \max\left(0, \sup_{0\leq s\leq t}-\left(\sqrt{x_0}+W_s\right) \right),
\end{align*}
where the second equality follows from Skorokhod's lemma, see \cite[Lem.\ III.6.14]{karatzas}.
Finally, using $\sup(-A)=-\inf(A)$ for $A\subseteq\R$ leads to the solution of \eqref{eq:bessel-1}
given by
\begin{align}\label{eq:bessel-1-sol}
	X_t = \left(
		 \left(W_t+\sqrt{x_0}\right)+\left(\inf_{0\leq s\leq t} W_s+\sqrt{x_0}\right)^-
	\right)^2,
\end{align}
where $x^-=-\min(0,x)$ denotes the negative part of $x$.
We stress that the explicit solution~\eqref{eq:bessel-1-sol} of
the SDE~\eqref{eq:bessel-1} holds for any Brownian motion $W$
and does not depend on the particular construction given
in \eqref{eq:def-W}, see \cite[Cor.\ V.3.23]{karatzas}, since
pathwise uniqueness and strong existence holds for the SDE~\eqref{eq:bessel-1}.
Hence the unique strong solution of the SDE~\eqref{eq:bessel-1} is given by \eqref{eq:bessel-1-sol}.

\begin{rem}
	It is well-known that the solution of the SDE~\eqref{eq:bessel-1} can be expressed in terms of
	$B$ in~\eqref{eq:bessel-1-sol-B}, see, e.g., \cite[Ex.\ IX.3.16]{yor}. However,
	we are not aware of a result regarding the explicit form of the strong solution
	given by \eqref{eq:bessel-1-sol}.
	
	In the context of SDEs, the somehow explicit solution of $X$ by means of $B$ in \eqref{eq:bessel-1-sol-B}
	is rather useless for strong approximation. The concept of strong solutions entails a functional
	dependence of the solution process and the input Brownian motion appearing in the SDE,
	which is $W$ in our case. We thus seek to ``construct'' the solution $X$ out of $W$.
\end{rem}

\begin{rem}\label{rem:dist}
	Equation~\eqref{eq:bessel-1-sol-B} clearly yields
	\begin{align*}
		\left( X_t \right)_{t\geq0} \eqdist \left( (W_t+\sqrt{x_0})^2 \right)_{t\geq0},
	\end{align*}
	cf.~\cite[Thm.~III.6.17]{karatzas}.
	However, this equation
	is only valid in the distributional sense and does not hold pathwise.
	Note that the Brownian path attains its
	running minimum whenever the solution hits $0$. More precisely, we have
	\begin{align}\label{eq:hit-0}
		X_t=0
			\quad\Leftrightarrow\quad
			W_t \leq -\sqrt{x_0} \ \wedge\  W_t = \inf_{0\leq s\leq t} W_s,
	\end{align}
	cf.~Figure~\ref{fig:path}.
\end{rem}

\begin{figure}[htbp]
	\centering
	\includegraphics[width=0.9\linewidth]{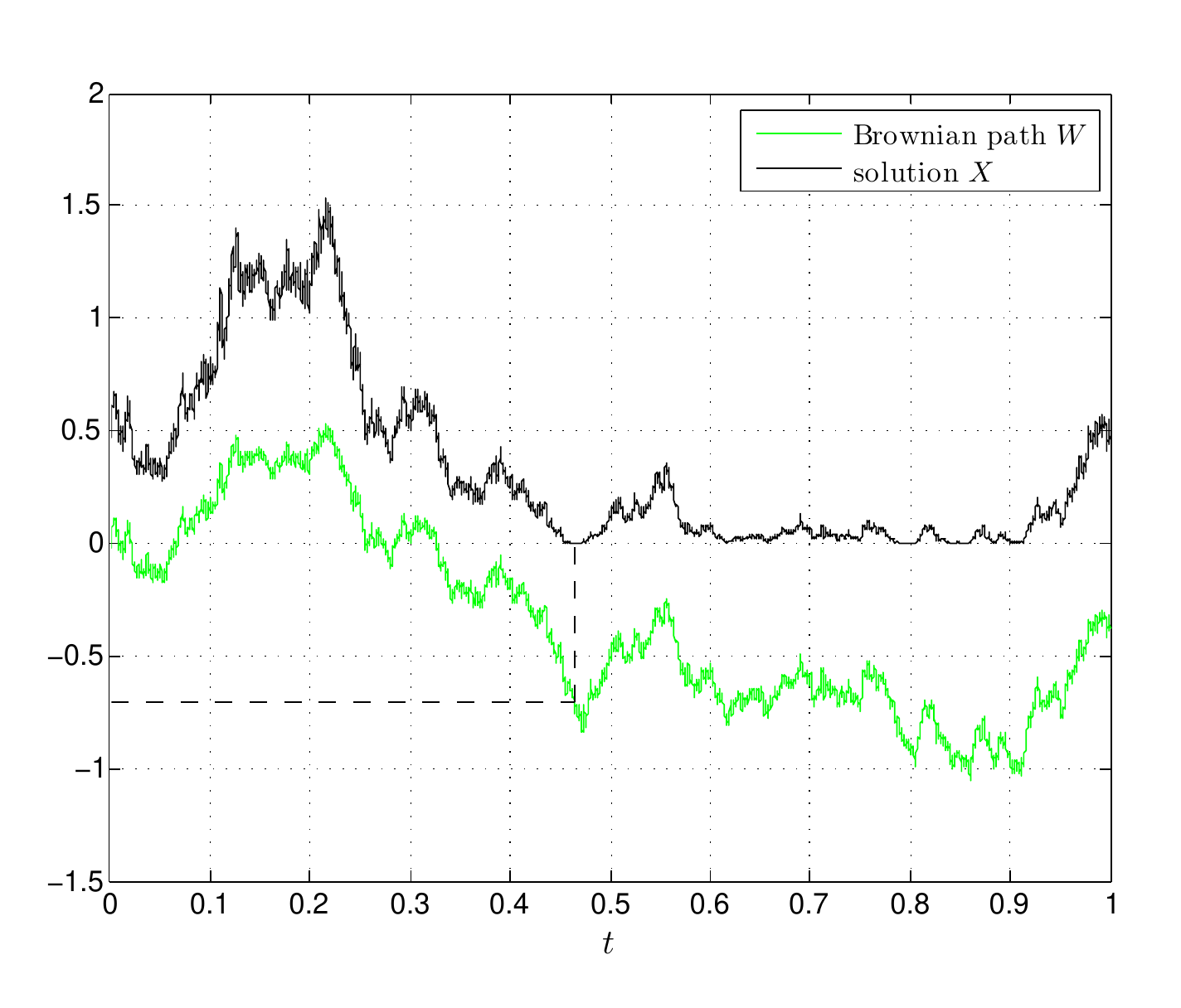}
	\caption{Brownian path and corresponding solution~\eqref{eq:bessel-1-sol}
		with initial condition $x_0=0.5$.
		The solution hits the boundary $0$ according to \eqref{eq:hit-0}.}
	\label{fig:path}
\end{figure}

\begin{rem}
	The expression
	\begin{align*}
		\left(W_t+\sqrt{x_0}\right)+\left(\inf_{0\leq s\leq t} W_s+\sqrt{x_0}\right)^-
	\end{align*}
	appearing in \eqref{eq:bessel-1-sol} is known as a reflected Brownian motion
	(with reflecting barrier at~$0$) starting in $\sqrt{x_0}\geq0$, see \cite{tanaka}.
\end{rem}

\bigskip
We now turn to the more general case of SDE~\eqref{eq:bessel-1-b}
with arbitrary $b\in\R$.

\begin{prop}\label{prop:sol-b}
	The unique strong solution of the SDE~\eqref{eq:bessel-1-b} is given by
	\begin{align*}
		X_t = \left(
				u_t + e^{-\frac{b}{2}t}
				\left( \inf_{0\leq s\leq t}\, e^{\frac{b}{2}s}\,u_s \right)^-
			\right)^2,
	\end{align*}
	where $u$ denotes the unique strong solution of the SDE
	\begin{align}\label{eq:linear-SDE-b}
		du_t = - \frac{b}{2}\,u_t \,dt + dW_t,
			\quad u_0=\sqrt{x_0}, \quad t\geq0.
	\end{align}
	In particular, the unique strong solution of the SDE~\eqref{eq:bessel-1}
	is given by \eqref{eq:bessel-1-sol}.
\end{prop}

\begin{rem}
	Note that the solution to the linear SDE~\eqref{eq:linear-SDE-b} is called
	Ornstein-Uhlenbeck process and can be solved explicitly by
	\begin{align}\label{eq:linear-SDE-b-sol}
		u_t = e^{-\frac{b}{2}t} \left(
				\sqrt{x_0} + \int_0^t e^{\frac{b}{2}s}\,dW_s
			\right), \quad t\geq0,
	\end{align}
	see, e.g., \cite[Ex.~V.6.8)]{karatzas}.
\end{rem}

\begin{proof}[Proof of Proposition~\ref{prop:sol-b}]
	Analogous to the above derivation, we start
	with a Brownian motion $B$ and consider the SDE
	\begin{align*}
		du_t^B = -\frac{b}{2}\,u^B_t \,dt + dB_t,
			\quad u_0^B=\sqrt{x_0}, \quad t\geq0.
	\end{align*}
	Moreover, we assume that the
	Brownian motion appearing in \eqref{eq:bessel-1-b} has the particular form
	\begin{align*}
		W_t = \int_0^t \sgn(u_s^B)\,dB_s, \quad t\geq0.
	\end{align*}
	Then, It\^{o}'s formula shows
	\begin{align*}
		d\left( \left(u_t^B\right)^2 \right)
			&= \left( 1-b\left(u_t^B\right)^2 \right) dt + 2\,u_t^B \,dB_t \\
			&= \left( 1-b\left(u_t^B\right)^2 \right) dt + 2 \abs{u_t^B} dW_t.
	\end{align*}
	Hence the solution of SDE~\eqref{eq:bessel-1-b} is given by
	\begin{align}\label{eq:proof-prop-bessel-1-b}
		X_t=\left(u_t^B\right)^2, \quad t\geq0.
	\end{align}
	On the other hand, the Tanaka-Meyer formula \cite[Thm.\ III.7.1(v)]{karatzas}
	applied to
	\begin{align*}
		\tilde{u}_t^B=u_t^B e^{\frac{b}{2}t}
	\end{align*}
	combined with the explicit expression \eqref{eq:linear-SDE-b-sol} yields
	\begin{align*}
		\abs{\tilde{u}_t^B}
			&= \sqrt{x_0} + \int_0^t e^{\frac{b}{2}s}\sgn\left(\tilde{u}_s^B\right) dB_s
			+ 2\Lambda_t^{\tilde{u}_B}(0) \\
		&= \sqrt{x_0} + \int_0^t e^{\frac{b}{2}s}\,dW_s
			+ 2\Lambda_t^{\tilde{u}_B}(0),
	\end{align*}
	where $\Lambda^{\tilde{u}_B}(0)$ denotes the semimartingale local time of $\tilde{u}^B$
	at $0$. Thus, Skorokhod's lemma \cite[Lem.\ III.6.14]{karatzas} shows
	\begin{align*}
		\abs{\tilde{u}_t^B} = \sqrt{x_0} + \int_0^t e^{\frac{b}{2}s}\,dW_s
			+ \left(\inf_{0\leq s\leq t} \sqrt{x_0} + \int_0^s e^{\frac{b}{2}u}\,dW_u\right)^-,
	\end{align*}
	and consequently
	\begin{align*}
		\abs{u_t^B} = u_t + e^{-\frac{b}{2}t}
			\left( \inf_{0\leq s\leq t}\, e^{\frac{b}{2}s}\,u_s \right)^-.
	\end{align*}
	It remains to apply \cite[Cor.\ V.3.23]{karatzas}.
\end{proof}

\begin{rem}\label{rem:reflected-SDE}
	Consider the situation of the proof of Proposition~\ref{prop:sol-b}.
	The Tanaka-Meyer formula applied to $u_t^B$ yields
	\begin{align*}
		\abs{u_t^B}
			&= \sqrt{x_0} -\frac{b}{2} \int_0^t u^B_s\sgn\left(u_s^B\right) ds
				+ \int_0^t \sgn\left(u_s^B\right) dB_s + 2\Lambda_t^{u_B}(0) \\
			&= \sqrt{x_0} -\frac{b}{2} \int_0^t \abs{u_s^B} ds
				+ \int_0^t dW_s + 2\Lambda_t^{u_B}(0).
	\end{align*}
	Hence $Z_t=\abs{u_t^B}$ is the solution of the reflected SDE on the domain $D={]0,\infty[}$
	given by
	\begin{align}\label{eq:refelcted-SDE}
		dZ_t = - \frac{b}{2}\,Z_t \,dt + dW_t + dK_t,
			\quad Z_0=\sqrt{x_0}, \quad t\geq0,
	\end{align}
	where $K$ is a process of bounded variation with variation increasing only when $Z_t\in\partial D=\{0\}$.
	For details on reflected SDEs we refer to~\cite{tanaka}.
	In view of~\eqref{eq:proof-prop-bessel-1-b}, we can express the solution $X$ to the SDE~\eqref{eq:bessel-1-b} by
	\begin{align*}
		X_t=(Z_t)^2, \quad t\geq0.
	\end{align*}
\end{rem}

\section{Equidistant Methods for SDE~(\ref{eq:bessel-1})}\label{sec:equi}
As shown in the previous section, the SDE~\eqref{eq:bessel-1}
admits the explicit solution \eqref{eq:bessel-1-sol}. This immediately links
the problem of approximating SDE \eqref{eq:bessel-1} to global optimization
under the Wiener measure. We refer to \cite{ritter1990} for results on
global optimization under the Wiener measure.

In this section we show sharp (up to constants) upper and lower bounds for $e_p^\equi$
for all $p\geq1$. Recall that $e_p^\equi$ denotes the $n$-th minimal error corresponding
to the approximation of SDE~\eqref{eq:bessel-1} and algorithms that are based on
equidistant grids, see~\eqref{eq:minimal-error-equi}.
We show that the convergence rate of $e_p^\equi$ is equal to $1/2$. Moreover, we show that
the drift-implicit Euler scheme converges with the optimal rate $1/2$.

For the proofs we exploit results from \cite{pitman} on the asymptotic error distribution
of the infimum of a Brownian motion approximated by equidistant points on the unit interval.
For $n\in\N$ we define
\begin{align}\label{eq:dn}
	\delta_n=\sqrt{n}\cdot \left(
			\min_{0\leq k\leq n} W_{\frac{k}{n}} - \inf_{0\leq s\leq 1} W_s
		\right),
\end{align}
where $W$ denotes a standard Brownian motion.
The following result is due to \cite[Thm.~1 and Lem.~6]{pitman}.

\begin{thm}[\cite{pitman}]\label{thm:pitman}
	\ 
	\begin{enumerate}[(i)]
	\item The sequence $(\delta_n)_{n\in\N}$ converges in distribution.
	\item For all $p\in[1,\infty[$ the sequence $(\delta_n^p)_{n\in\N}$ is uniformly integrable. In particular,
		for all $p\in[1,\infty[$ we have
		\begin{align}\label{eq:min-upper}
			\left(
				\E{ \abs{ \min_{0\leq k\leq n} W_{\frac{k}{n}} - \inf_{0\leq s\leq 1} W_s }^p }
			\right)^{1/p}
				\preccurlyeq n^{-1/2}.
		\end{align}
	\end{enumerate}
\end{thm}

\begin{rem}
	In \cite{pitman}, the limiting distribution of the sequence $(\delta_n)_{n\in\N}$
	is given explicitly by means of three-dimensional Bessel processes.
\end{rem}

Recall that the solution of~\eqref{eq:bessel-1} is given by $X_1=\left(Y_1\right)^2$ with
\begin{align}\label{eq:Y1}
	Y_1 = \left(W_1+\sqrt{x_0}\right)+\left(\inf_{0\leq s\leq 1} W_s+\sqrt{x_0}\right)^-.
\end{align}
Moreover, for $n\in\N$ we define the approximation $\Xeq_1$ of $X_1$ by
\begin{align*}
	\Xeq_1=\left(\Yeq_1\right)^2,
\end{align*}
where
\begin{align*}
	\Yeq_1 = \left( W_1+\sqrt{x_0} \right)
		+ \left( \min_{0\leq k\leq n} W_{\frac{k}{n}} + \sqrt{x_0} \right)^-
\end{align*}
serves as an approximation of $Y_1$. Here, the global infimum is simply replaced by
the discrete minimum over $n$ equidistant knots.

The following upper bound is a consequence of Theorem~\ref{thm:pitman}(ii).

\begin{thm}\label{thm:upper-bound}
	For every $1\leq p<\infty$ we have
	\begin{align*}
		e_p\left( \Xeq_1 \right) \preccurlyeq n^{-1/2}.
	\end{align*}
	In particular, the $n$-th minimal error satisfies
	\begin{align*}
		e_p^\equi(n) \preccurlyeq n^{-1/2}.
	\end{align*}
\end{thm}

\begin{proof}
	At first, note that $0 \leq \Yeq_1 \leq Y_1$ and
	\begin{align*}
		\abs{ Y_1 - \Yeq_1 }
			&= \left( \inf_{0\leq s\leq 1} W_{s} + \sqrt{x_0} \right)^-
				-\left( \min_{0\leq k\leq n} W_{\frac kn} + \sqrt{x_0} \right)^-
			\leq \min_{0\leq k\leq n} W_{\frac kn} - \inf_{0\leq s\leq 1} W_{s}.
	\end{align*}
	Hence we get
	\begin{align*}
		\abs{ X_1-\Xeq_1 } = \left( Y_1+\Yeq_1 \right) \cdot \abs{ Y_1-\Yeq_1 }
			\leq 2 Y_1 \cdot
				\left(
					\min_{0\leq k\leq n} W_{\frac kn}-\inf_{0\leq s\leq 1} W_{s}
				\right).
	\end{align*}
	Finally, Cauchy-Schwarz inequality yields
	\begin{align*}
		\E{\abs{ X_1-\Xeq_1 }^p}
			&\leq 2^p\cdot \left(
				\E{ Y_1^{2p} }
				\cdot \E{\abs{ \min_{0\leq k\leq n} W_{\frac kn}-\inf_{0\leq s\leq 1} W_{s} }^{2p}}
			\right)^{1/2} \\
		&\preccurlyeq n^{-p/2}
	\end{align*}
	due to \eqref{eq:min-upper} and
	\begin{align*}
		\E{ Y_1^r } < \infty, \quad 1\leq r<\infty,
	\end{align*}
	since $Y_1\eqdist \abs{W_1+\sqrt{x_0}}$, see Remark~\ref{rem:dist}.
\end{proof}

The natural extension of $\Xeq_1$ to an approximation on the whole equidistant grid with
mesh size $1/n$ is given by
\begin{align*}
	\Xeq_{\frac kn} = \left(\Yeq_{\frac kn}\right)^2, \quad k=0,\ldots,n,
\end{align*}
where
\begin{align*}
	\Yeq_{\frac kn} = \left( W_{\frac kn}+\sqrt{x_0} \right)
		+ \left( \min_{0\leq l\leq k} W_{\frac{l}{n}} + \sqrt{x_0} \right)^-,
		\quad k=0,\ldots,n.
\end{align*}
Let us stress that this scheme can be expressed by the following Euler-type scheme
\begin{align}\label{eq:euler-projected}
	\Yeq_{\frac{k+1}{n}} = \max\left(
			\Yeq_{\frac kn} + \left(W_{\frac{k+1}{n}}-W_{\frac kn}\right), 0
		\right), \quad k=0,\dots,n-1,
\end{align}
and $\Yeq_0=\sqrt{x_0}$.
In the context of reflected SDEs, scheme~\eqref{eq:euler-projected} is simply
the projected Euler scheme for a reflected Brownian motion. Due to
\begin{align*}
	\max(x,0) = \frac{x+\abs{x}}{2} = \frac{x+\sqrt{x^2}}{2}, \quad x\in\R,
\end{align*}
the drift-implicit Euler scheme~\eqref{eq:drift-implicit-Euler-scheme}
for $b=0$ reads
\begin{align}\label{eq:drift-implicit-Euler-scheme-2}
	\Ximp_\frac{k+1}{n} = \left( \max \left(
			\sqrt{ \Ximp_\frac{k}{n} } + \left(W_{\frac{k+1}{n}}-W_{\frac kn}\right), 0
		\right) \right)^2, \quad k=0,\ldots,n-1.
\end{align}
Thus it coincides with the squared version of \eqref{eq:euler-projected}, i.e.,
\begin{align*}
	\Ximp_\frac{k}{n} = \Xeq_\frac{k}{n}, \quad k=0,\ldots,n.
\end{align*}

\begin{rem}\label{rem:sup-norm}
	We briefly comment on results for the more general case
	of SDE~\eqref{eq:bessel-1-b} with arbitrary $b\in\R$.
	Moreover, we consider a stronger global error criterion
	where pathwise the global error is measured in the supremum norm.
	Up to a logarithmic factor we will obtain the same error bound
	as in Theorem~\ref{thm:upper-bound}.
	
	For $n\in\N$ we denote by $\Zbar=(\Zbar)_{0\leq t\leq 1}$ the projected Euler scheme
	with $n$ steps associated to the reflected SDE~\eqref{eq:refelcted-SDE}
	up to time $t=1$, see \cite{slominski}. More precisely, $\Zbar$ is defined by
	\begin{align*}
		\Zbar_{0} &= \sqrt{x_0},\\
		\Zbar_{\frac{k+1}{n}} &= \max \left(
				\Zbar_{\frac{k}{n}} - \frac{b}{2}\cdot\Zbar_{\frac{k}{n}}\cdot\frac{1}{n}
				+ \left(W_{\frac{k+1}{n}}-W_{\frac kn}\right), 0
			\right),
			\quad k=0,\dots,n-1,
	\end{align*}
	and piecewise constant interpolation, i.e.,
	\begin{align*}
	\Zbar_{t} = \Zbar_{\frac{k}{n}},
		\quad t\in{[k/n,(k+1)/n[}\,,
		\quad k=0,\dots,n-1.
	\end{align*}
	The solution $X$ to SDE~\eqref{eq:bessel-1-b} is then approximated by
	\begin{align*}
		\Xbar_t = \left( \Zbar_{t} \right)^2, \quad t\in[0,1].
	\end{align*}
	At the grid points, this scheme coincides with the drift-implicit Euler
	scheme~\eqref{eq:drift-implicit-Euler-scheme-2} if $b=0$.
	Similar to the proof of Theorem~\ref{thm:upper-bound}, we obtain
	\begin{align*}
		\left( \E{ \sup_{0\leq t\leq 1} \abs{X_t-\Xbar_t}^p } \right)^{1/p}
			\preccurlyeq n^{-1/2}\cdot \sqrt{\ln(1+n)}
	\end{align*}
	for all $p\in{[1,\infty[}$ due to \cite[Cor.~2.5, Cor.~2.6 and Thm.~3.2(i)]{slominski} and
	Remark~\ref{rem:reflected-SDE}.
	Let us stress that this error bound is the first strong
	convergence result with a positive rate in the case $\sigma^2/(2a)=2$,
	cf.~Figure~\ref{fig:CIR-rates}.
\end{rem}

We now turn to the question whether an algorithm can do better
than $\Xeq_1$ in an asymptotic sense if this algorithm has
the same information about the Brownian motion as $\Xeq_1$.

Recall the definition of the $n$-th minimal error $e_p^\equi$
given in~\eqref{eq:minimal-error-equi}. The proof of the following
theorem is postponed to Section~\ref{sec:proofs}.

\begin{thm}\label{thm:lower-bound}
	For all $p\in{[1,\infty[}$ we have
	\begin{align*}
		e^\equi_p(n) \succcurlyeq n^{-1/2}.
	\end{align*}
\end{thm}

Combining Theorem~\ref{thm:upper-bound} and Theorem~\ref{thm:lower-bound} yields
the following asymptotic behavior of the $n$-th minimal error.

\begin{cor}\label{cor:convergence-equi}
	For all $p\in{[1,\infty[}$ we have
	\begin{align*}
		e^\equi_p(n) \asymp n^{-1/2}.
	\end{align*}
	In particular, the drift-implicit Euler scheme~\eqref{eq:drift-implicit-Euler-scheme-2}
	is asymptotically optimal.
\end{cor}

\begin{rem}
	In Corollary~\ref{cor:convergence-equi} we obtain the same rate
	as in~\cite{ritter1990} for global optimization.
	In \cite{ritter1990}, the author studies optimal approximation of
	the time point where a Brownian motion attains its maximum,
	and provides a detailed analysis of general non-adaptive algorithms
	that do not necessarily rely on equidistant grids.
\end{rem}

\begin{rem}
	If we allow for more information about the Brownian path than just
	point evaluations $W_\frac{1}{n},\ldots,W_1$, the situation may change completely.
	For instance, if we consider algorithms that have access to the final value $W_1$ and to the infimum
	$\inf_{0\leq s\leq 1} W_s$ of the Brownian path, the problem of strong approximation
	of $X_1$ becomes trivial.
	Let us stress that the joint distribution of $\left( \inf_{0\leq s\leq1} W_s, W_1 \right)$
	has an explicit representation by means of a Lebesgue density,
	see, e.g., \cite[p.~154]{handbook-bm}.
\end{rem}

\section{Adaptive Methods for SDE~(\ref{eq:bessel-1})}\label{sec:adaptive}
In this section we present an adaptive algorithm for the approximation
of the solution of SDE~\eqref{eq:bessel-1} based on sequential observations
$W_{t_1},W_{t_2},\ldots$ of the Brownian motion $W$. In contrast to
Section~\ref{sec:equi}, here the points $t_1,t_2,\ldots$ are chosen adaptively,
i.e., the $k$-th evaluation site $t_k$ is a measurable function of
$W_{t_1}, \ldots, W_{t_{k-1}}$.
From Proposition~\ref{prop:sol-b} it is clear that the actual task consists
of the approximation of the global infimum $\inf_{s\in[0,1]} W_s$.
For this we use the adaptive algorithm from \cite{chh2015},
see also~\cite{Calvin1997,Calvin2001}.
This algorithm approximates $\inf_{s\in[0,1]} W_s$ by
the discrete minimum $\min_{0\leq k\leq n}W_{t_k}$.
In the following we describe the adaptive choice of $t_1,t_2,\ldots,t_n$.

\bigskip
The first observation is non-adaptively chosen to be the endpoint,
i.e., $t_1=1$. Moreover, for notational convenience we define $t_0=0$.

Let $n\in\N$, and consider the $(n+1)$-th step of the algorithm
where $t_{n+1}$ will be chosen based on the previous observations
$W_{t_1}, \ldots, W_{t_{n}}$.
For this, we denote the ordered first $n$ evaluation sites by
\begin{align*}
	0 = t_0\n < t_1\n < \ldots < t_n\n = 1,
\end{align*}
such that $\{t_0,\ldots,t_n\}=\{t_0\n,\ldots,t_n\n\}$.
Moreover, we assume that we have made the following observations
\begin{align}\label{eq:w-cond}
	W_{t_0\n}=y_0\n=0,\quad 
	W_{t_1\n}=y_1\n,\quad\dots,\quad 
	W_{t_n\n}=y_n\n,
\end{align}
and we denote the corresponding discrete minimum by
\begin{align*}
	m\n = \min_{0\leq k\leq n} y_k\n.
\end{align*}
Conditioned on \eqref{eq:w-cond},
we have independent Brownian bridges from $y_{k-1}\n$ to $y_k\n$ on
the subinterval $[t_{k-1}\n,t_k\n]$, for $k\in\{1,\ldots,n\}$.
In the following we denote a Brownian bridge from $x$ to $y$ on $[0,T]$
with $x,y\in\R$ and $T>0$ by $B^{x,T,y}$.

The basic idea of the adaptive algorithm is a simple greedy strategy:
The next observation is taken at the midpoint of
the subinterval where the probability that the corresponding Brownian bridge
undershoots the current discrete minimum minus some threshold $\eps\n>0$ is maximal.
More precisely, we split the interval according to
\begin{align}\label{eq:adap-crit}
	k^* = \argmax_{1\leq k\leq n}\ \prob\left(
			\inf_{\ 0\leq s\leq T_k\n} B_s^{y_{k-1}\n,T_k\n,y_k\n}
			\leq m\n-\eps\n
		\right)
\end{align}
with $T_k\n=t_k\n-t_{k-1}\n$, and evaluate $W$ at
\begin{align*}
	t_{n+1} = \left( t_{k^*}\n + t_{k^*-1}\n \right)/2.
\end{align*}
We note that the infimum of a Brownian bridge
$(B_t^{x,T,y})_{t\in[0,T]}$ satisfies
\begin{align*}
	\prob \left( \inf_{0\leq s\leq T}B_s^{x,T,y} < z \right)
		&= \exp\left(-\frac{2}{T}\,(x-z)(y-z) \right)
\end{align*}
for $z\leq\min(x,y)$, see \cite[p.\ 67]{handbook-bm}.
Hence \eqref{eq:adap-crit} reduces to maximizing
\begin{align*}
	k^* = \argmax_{1\leq k\leq n}\ \frac{ T_k\n }{
			\left( y_{k-1}\n-m\n+\eps\n \right) \cdot
			\left( y_{k}\n-m\n+\eps\n \right)
		}.
\end{align*}
Finally, we have to specify the threshold $\eps\n$.
This threshold is chosen to be
\begin{align*}
	\eps\n = \sqrt{\lambda\, h\n\, \ln(1/h\n)},
\end{align*}
where $h\n = \min_{1\leq k\leq n} T_k\n$ denotes
the length of the smallest subinterval at step $n$ and where ${\lambda\in{[1,\infty[}}$
is some prespecified parameter. Let us stress that
all above (adaptive) quantities depend on the choice of the
parameter $\lambda$, although it is not explicitly indicated.

This amounts to a family of adaptive algorithms defined by
\begin{align}\label{eq:adap-alg}
	\Xhat\n_{\adap,\lambda} =
		\left(
			W_{t_1}+\sqrt{x_0} +
			\left( \min_{0\leq k\leq n}W_{t_k}+\sqrt{x_0} \right)^-
		\right)^2
\end{align}
for the approximation of~\eqref{eq:bessel-1-sol},
where the adaptively chosen points $t_1,\ldots,t_n$ depend
on the prespecified choice of $\lambda\in{[1,\infty[}$.

\begin{rem}
	A straightforward implementation of the algorithm~\eqref{eq:adap-alg} on a computer
	requires operations of order $n^2$.
\end{rem}

The following result is an immediate consequence of~\cite[Thm.~1]{chh2015}.

\begin{thm}\label{thm:adaptive}
	For all $p\in{[1,\infty[}$ and for all $q\in{[1,\infty[}$
	there exists $\lambda\in{[1,\infty[}$ such that
	\begin{align*}
		e_p\left( \Xhat\n_{\adap,\lambda} \right) \preccurlyeq n^{-q}.
	\end{align*}
\end{thm}

\begin{rem}
	The analysis in \cite{chh2015} shows that
	\begin{align*}
		\lambda \geq 144\cdot (1+2pq)
	\end{align*}
	is sufficient to obtain a convergence order $q\in{[1,\infty[}$ for
	the $L_p$-norm in Theorem~\ref{thm:adaptive}.
	However, numerical experiments indicate an exponential decay
	even for small values of $\lambda$,
	see Figure~\ref{fig:numerical-results}.
\end{rem}

Recall the definition of the $n$-th minimal error $e_p^\adap$
given in~\eqref{eq:minimal-error-adap}.

\begin{cor}\label{cor:convergence-adap}
	For all $p\in{[1,\infty[}$ and for all $q\in{[1,\infty[}$ we have
	\begin{align*}
		e_p^\adap(n) \preccurlyeq n^{-q}.
	\end{align*}
\end{cor}

\begin{figure}[htp]
	\centering
	\includegraphics[width=0.9\linewidth]{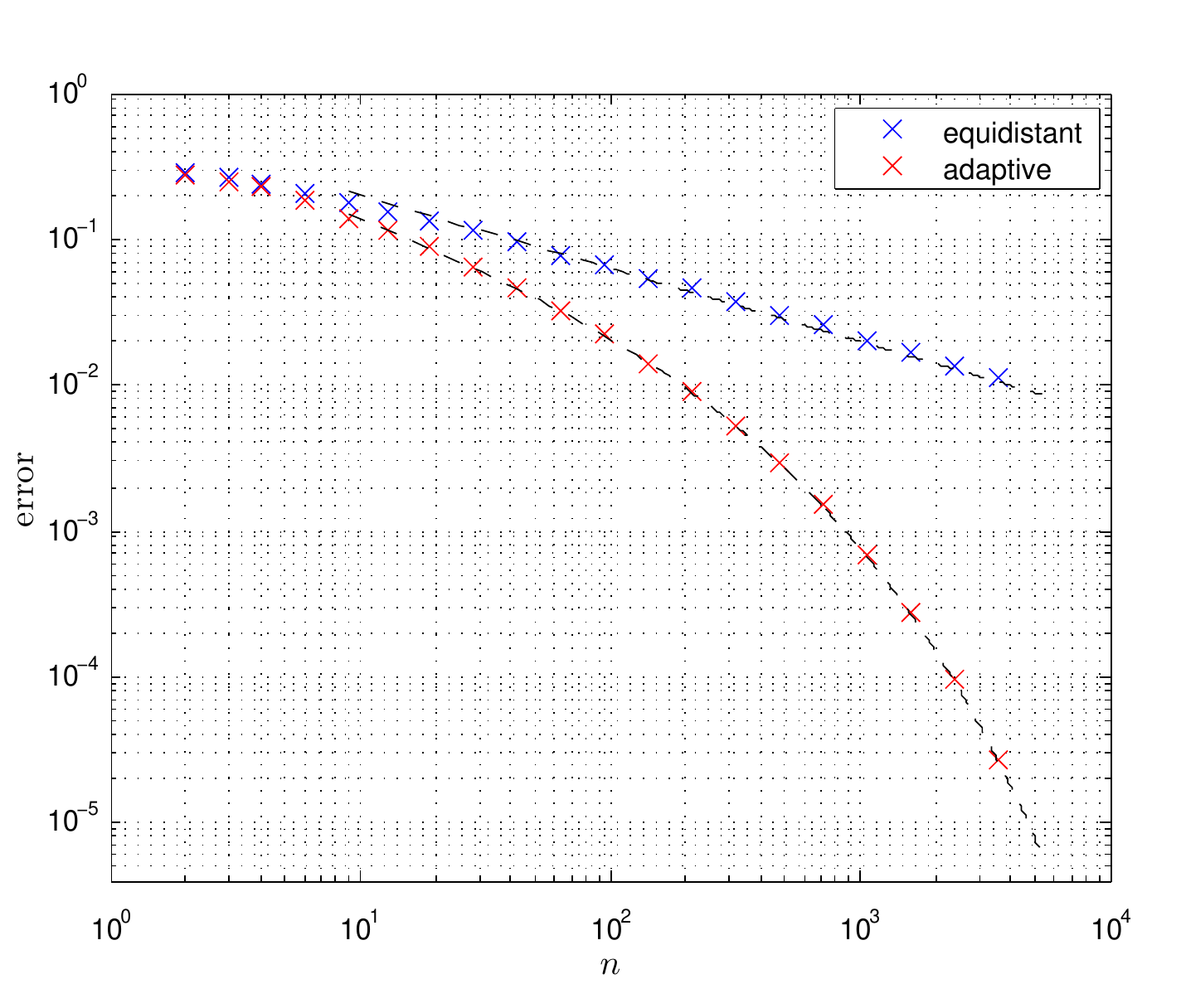}
	\caption{Numerical results for the drift-implicit Euler
		scheme~\eqref{eq:drift-implicit-Euler-scheme-2} (blue)
		and the adaptive algorithm~\eqref{eq:adap-alg} with $\lambda=4$ (red).
		The error given by $e_p(\cdot)$ with $p=2$ is estimated based on $10^4$ samples.
		The dashed lines show $0.63\cdot n^{-1/2}$ and $0.64\cdot n^{-0.54}\exp(-0.095\sqrt{n})$.
		}
	\label{fig:numerical-results}
\end{figure}

\section{Proofs}\label{sec:proofs}
In this section we provide the proof of Theorem~\ref{thm:lower-bound}, which
relies on Lemma~\ref{lem:1} and Lemma~\ref{lem:2}.

For $n\in\N$ we define $i_n^\ast\in\{0,\dots,n\}$ to be an index that satisfies
\begin{align*}
	W_{\frac{i_n^\ast}{n}} = \min_{0\leq i\leq n} W_{\frac{i}{n}}.
\end{align*}
Note that $i_n^\ast$ is (almost surely) uniquely defined.

\begin{lem}\label{lem:inf}
For $z\geq 0$ we have
	\begin{align*}
		\inf_{n\in\N} \prob \left(
				\abs{W_\frac{i_n^\ast}{n}-W_\frac{i_n^\ast+1}{n}} \leq \frac{1}{\sqrt{n}},\ 
				W_\frac{i_n^\ast}{n}\leq -z
			\right) > 0.
	\end{align*}
\end{lem}
\begin{proof}
	Let $n\in\N$.
	Due to conditional independence we get
	\begin{align*}
		\prob &\left(
				\abs{W_\frac{i_n^\ast}{n}-W_\frac{i_n^\ast+1}{n}} \leq \frac{1}{\sqrt{n}},\ 
				W_\frac{i_n^\ast}{n}\leq -z
			\right) \\
		&= \sum_{i=0}^n \prob \left(
				\abs{W_\frac{i_n^\ast}{n}-W_\frac{i_n^\ast+1}{n}} \leq \frac{1}{\sqrt{n}},\ 
				W_\frac{i_n^\ast}{n}\leq -z
				\cond {i_n^\ast=i}
			\right)
			\cdot \prob\left( i_n^\ast=i \right) \\
		&= \sum_{i=0}^n \prob \left(
				\abs{W_\frac{i_n^\ast}{n}-W_\frac{i_n^\ast+1}{n}} \leq \frac{1}{\sqrt{n}}
				\cond {i_n^\ast=i}
			\right)
			\cdot \prob \left(
				W_\frac{i_n^\ast}{n}\leq -z
				\cond {i_n^\ast=i}
			\right)
			\cdot \prob\left( i_n^\ast=i \right).
	\end{align*}
	Moreover, straightforward calculations show
	\begin{align*}
		\prob \left(
				\abs{W_\frac{i_n^\ast}{n}-W_\frac{i_n^\ast+1}{n}} \leq \frac{1}{\sqrt{n}}
				\cond {i_n^\ast=i}
			\right)
			&= \prob \left(
				\abs{W_{\frac{1}{n}}} \leq \frac{1}{\sqrt{n}}
				\cond W_{\frac{1}{n}} \geq0,\ \ldots,\ W_{\frac{n-i}{n}} \geq0
			\right) \\
		&= \prob \left(
				\abs{W_{1}} \leq 1 \,\Big\vert\, W_{1} \geq0,\ \ldots,\ W_{n-i} \geq0
			\right) \\
		&\geq \prob(0\leq W_1\leq1)
	\end{align*}
	for $i\in\{0,\dots,n\}$. Thus we have
	\begin{align*}
		\prob &\left(
				\abs{W_\frac{i_n^\ast}{n}-W_\frac{i_n^\ast+1}{n}} \leq \frac{1}{\sqrt{n}},\ 
				W_\frac{i_n^\ast}{n}\leq -z
			\right) \\
		&\hspace{1cm}\geq \prob(0\leq W_1\leq1)
			\cdot \prob \left(
				W_\frac{i_n^\ast}{n}\leq -z
			\right)
			\geq \prob(0\leq W_1\leq1) \cdot \prob\left( W_1\leq -z\right) > 0,
	\end{align*}
	which completes the proof.
\end{proof}

\begin{lem}\label{lem:1}
	There exist $n_0\in\N$ and $0<\eps_0<1$ such that
	\begin{align*}
		\prob \left( \abs{ W_\frac{i_n^\ast}{n}-W_\frac{i_n^\ast+1}{n} } \leq \frac{1}{\sqrt{n}},\
			W_\frac{i_n^\ast}{n}\leq -\sqrt{x_0},\
			W_\frac{i_n^\ast}{n}\leq W_1-\eps_0\right)\geq \eps_0
	\end{align*}
	for all $n\geq n_0$.
\end{lem}

Let us mention that the events
\begin{align*}
	W_\frac{i_n^\ast}{n}\leq -\sqrt{x_0} \quad\wedge\quad W_\frac{i_n^\ast}{n}\leq W_1-\eps_0
\end{align*}
simply ensure that reflection occurs and that the discrete minimum (and thus the global infimum)
is not attained at the final time point $t=1$.

\begin{proof}[Proof of Lemma~\ref{lem:1}]
	Donsker's invariance principle and \cite[Thm.\ 2.7]{billingsley} yield
	\begin{align*}
		\left( W_\frac{i_n^\ast}{n} - W_1 \right)_n \todist \inf_{0\leq s\leq 1} W_s -W_1.
	\end{align*}
	Thus the portmanteau theorem implies
	\begin{align*}
		\liminf_{n\to\infty} \prob \left(
				W_\frac{i_n^\ast}{n} - W_1 \leq -\eps
			\right)
		&\geq \prob \left(
				\inf_{0\leq s\leq 1} W_s -W_1 < -\eps
			\right)
	\end{align*}
	for $\eps>0$.
	Moreover, since $\prob\left(\inf_{0\leq s\leq 1} W_s -W_1 <0\right)=1$,
	there exists $\eps_0>0$ such that
	\begin{align*}
		\liminf_{n\to\infty} \prob& \left(
				W_\frac{i_n^\ast}{n} - W_1 \leq -\eps_0
			\right)\\
			&\geq 1 - \frac{1}{2}\cdot \inf_{n\in\N} \prob \left(
				\abs{W_\frac{i_n^\ast}{n}-W_\frac{i_n^\ast+1}{n}} \leq \frac{1}{\sqrt{n}},\ 
				W_\frac{i_n^\ast}{n}\leq -\sqrt{x_0}
			\right)
	\end{align*}
	due to Lemma~\ref{lem:inf}. Hence we get
	\begin{align*}
		&\liminf_{n\to\infty} \prob \left(
				\abs{ W_\frac{i_n^\ast}{n}-W_\frac{i_n^\ast+1}{n} } \leq \frac{1}{\sqrt{n}},\ 
				W_\frac{i_n^\ast}{n}\leq -\sqrt{x_0},\ 
				W_\frac{i_n^\ast}{n}\leq W_1-\eps_0
			\right) \\
		&\geq \liminf_{n\to\infty} \left(
				\prob \left(
						\abs{ W_\frac{i_n^\ast}{n}-W_\frac{i_n^\ast+1}{n} } \leq \frac{1}{\sqrt{n}},\ 
						W_\frac{i_n^\ast}{n}\leq -\sqrt{x_0}
					\right) 
				+ \prob \left( W_\frac{i_n^\ast}{n}\leq W_1-\eps_0 \right)
				- 1
			\right) \\
		&\geq \inf_{n\in\N}
			\prob \left(
					\abs{ W_\frac{i_n^\ast}{n}-W_\frac{i_n^\ast+1}{n} } \leq \frac{1}{\sqrt{n}},\ 
					W_\frac{i_n^\ast}{n}\leq -\sqrt{x_0}
			\right) 
			+ \liminf_{n\to\infty} \prob \left( W_\frac{i_n^\ast}{n}\leq W_1-\eps_0 \right)
			- 1 \\
		&\geq \frac{1}{2}\cdot \inf_{n\in\N}
			\prob \left(
					\abs{ W_\frac{i_n^\ast}{n}-W_\frac{i_n^\ast+1}{n} } \leq \frac{1}{\sqrt{n}},\ 
					W_\frac{i_n^\ast}{n}\leq -\sqrt{x_0}
			\right) > 0
	\end{align*}
	due to Lemma~\ref{lem:inf}.
\end{proof}

\begin{lem}\label{lem:2}
	Let $0<\eps_0<1$ be according to Lemma~\ref{lem:1}. Then there exists a constant $c_0>0$ such that
	\begin{align*}
		\prob \left( \abs{ X_1-\Xeq_1 } \leq \frac{c_0}{2\sqrt{n}} \right)
			\geq 1-\frac{\varepsilon_0}{4}
	\end{align*}
	for all $n\in\N$.
\end{lem}

\begin{proof}
	According to Theorem~\ref{thm:pitman}(i) and the portmanteau theorem, we have
	\begin{align*}
		\liminf_{n\to\infty} \prob \left( \delta_n\leq c \right)
						\geq \prob(\delta < c)
			\end{align*}
	for all $c\in\R$, where $\delta$ denotes the limit of $(\delta_n)_{n\in\N}$ given by~\eqref{eq:dn}.
	In particular, there exists $c>0$ such that
	\begin{align*}
		\prob\left(
				\min_{0\leq k\leq n} W_{\frac kn}-\inf_{0\leq s\leq 1} W_{s} \leq \frac{c}{\sqrt{n}},\ 
				{Y_1}\leq c
			\right)
			\geq 1 - \frac{\eps_0}{4}
	\end{align*}
	for all $n\in\N$ and $Y_1$ given by~\eqref{eq:Y1}. Finally, noting that
	\begin{align*}
		\abs{ X_1-\Xeq_1 }
			\leq 2 Y_1 \cdot
				\left(
					\min_{0\leq k\leq n} W_{\frac kn}-\inf_{0\leq s\leq 1} W_{s}
				\right),
	\end{align*}
	cf.~Theorem~\ref{thm:upper-bound}, implies
	\begin{align*}
		\set{ \min_{0\leq k\leq n} W_{\frac kn}-\inf_{0\leq s\leq 1} W_{s} \leq \frac{c}{\sqrt{n}},\
			 {Y_1}\leq c }
			\subseteq
			\set{ \abs{ X_1-\Xeq_1 } \leq \frac{c_0}{2\sqrt{n}} }
	\end{align*}
	with $c_0=4c^2$.
\end{proof}

\begin{proof}[Proof of Theorem~\ref{thm:lower-bound}]
	At first, note that Jensen's inequality implies
	\begin{align*}
		e^\equi_1(n) \leq e^\equi_p(n)
	\end{align*}
	for all $n\in\N$ and $p\in{[1,\infty[}$. Thus it suffices to consider $p=1$, i.e.,
	we will show the existence of a constant $c_1>0$ such that
	\begin{align*}
		\E{ \abs{ X_1-\Xhat_1 } } \geq c_1\cdot n^{-1/2}
	\end{align*}
	for all $n\in\N$ and for all random variables $\Xhat_1$ that are measurable
	w.r.t.~the $\sigma$-algebra $\A_n$ generated by $W_{\frac 1n},W_{\frac 2n},\ldots,W_1$.
	
	Let $n_0\in\N$, $\eps_0>0$, and $c_0>0$ be according to Lemma~\ref{lem:1}
	and Lemma~\ref{lem:2}, respectively. Without loss of generality,
	we may assume that $n\geq n_0$.
	In the following, we consider two cases separately.
	
	\bigskip
	At first, suppose that
	\begin{align*}
		\prob \left( \abs{ \Xhat_1 - \Xeq_1 } > \frac{c_0}{\sqrt{n}} \right)
			\geq \frac{\eps_0}{2}.
	\end{align*}
	By using reverse triangle inequality
	\begin{align*}
		\abs{ X_1-\Xhat_1 } \geq \abs{ \Xeq_1-\Xhat_1 } - \abs{ X_1-\Xeq_1 },
	\end{align*}
	we get
	\begin{align*}
		\left\{ \abs{ \Xeq_1-\Xhat_1 } > \frac{c_0}{\sqrt{n}} \right\}
			\cap \left\{ \abs{ X_1-\Xeq_1 } \leq \frac{c_0}{2\sqrt{n}} \right\}
			\subseteq \left\{ \abs{ X_1-\Xhat_1 } > \frac{c_0}{2\sqrt{n}} \right\},
	\end{align*}
	and thus
	\begin{align*}
		\prob \left( \abs{ X_1-\Xhat_1 } > \frac{c_0}{2\sqrt{n}} \right)
			&\geq \prob \left(
				\abs{ \Xeq_1-\Xhat_1 } > \frac{c_0}{\sqrt{n}},\
				\abs{ X_1-\Xeq_1 } \leq \frac{c_0}{2\sqrt{n}}
			\right) \\
		&\geq \frac{\eps_0}{2} + \left( 1-\frac{\eps_0}{4} \right) -1
	\end{align*}
	due to Lemma~\ref{lem:2}. This yields
	\begin{align}\label{eq:case-1}
		\E{ \abs{ X_1-\Xhat_1 } }
			\geq \frac{c_0}{2\sqrt{n}}\cdot \prob \left( \abs{ X_1-\Xhat_1 } > \frac{c_0}{2\sqrt{n}} \right)
			\geq \frac{c_0\,\eps_0}{8}\cdot\frac{1}{\sqrt{n}}.
	\end{align}
	
	\bigskip
	Now suppose that
	\begin{align*}
		\prob \left( \abs{ \Xhat_1 - \Xeq_1 } \leq \frac{c_0}{\sqrt{n}} \right)
			> 1 - \frac{\eps_0}{2},
	\end{align*}
	and define
	\begin{align*}
		A_n = \left\{ \left|W_\frac{i_n^\ast}{n}-W_\frac{i_n^\ast+1}{n}\right|\leq \frac{1}{\sqrt{n}},\
			W_\frac{i_n^\ast}{n}\leq -\sqrt{x_0},\
			W_\frac{i_n^\ast}{n}\leq W_1-\eps_0,\
			\abs{ \Xhat_1-\Xeq_1 } \leq \frac{c_0}{\sqrt{n}} \right\}.
	\end{align*}
	Let us stress that $A_n\in\A_n$ and
	\begin{align*}
		\prob(A_n) > \eps_0 + \left(1-\frac{\eps_0}{2}\right) - 1 = \frac{\eps_0}{2}
	\end{align*}
	due to Lemma~\ref{lem:1}. Moreover, we observe that reverse triangle inequality
	\begin{align*}
		\abs{ X_1-\Xhat_1 } \geq \abs{ X_1-\Xeq_1 } - \abs{ \Xeq_1-\Xhat_1 }
	\end{align*}
	yields
	\begin{align*}
		\left\{ \abs{ X_1-\Xeq_1 } \geq \frac{2c_0}{\sqrt{n}} \right\}
			\cap \left\{ \abs{ \Xhat_1-\Xeq_1 } \leq \frac{c_0}{\sqrt{n}} \right\}
			\subseteq \left\{ \abs{ X_1-\Xhat_1 } \geq \frac{c_0}{\sqrt{n}} \right\}.
	\end{align*}
	Combining this with
	\begin{align*}
		\abs{ X_1-\Xeq_1 }
			&= \left( W_\frac{i_n^\ast}{n} - \inf_{0\leq s\leq1} W_s \right)
			\cdot \left( W_1 - W_\frac{i_n^\ast}{n} + W_1 - \inf_{0\leq s\leq1} W_s \right) \\
		&\geq \left( W_\frac{i_n^\ast}{n} - \inf_{0\leq s\leq1} W_s \right)
			\cdot 2\eps_0
	\end{align*}
	on $A_n$, we obtain
	\begin{align*}
		A_n\cap
			\left\{  W_\frac{i_n^\ast}{n} - \inf_{0\leq s\leq1} W_s \geq \frac{c_0}{\eps_0\sqrt{n}} \right\}
			\subseteq \left\{ \abs{ X_1-\Xhat_1 } \geq \frac{c_0}{\sqrt{n}} \right\}. 
	\end{align*}
	Furthermore, we have
	\begin{align*}
		\inf_{0\leq s\leq1} W_s \leq W_{ \frac{i_n^\ast}{n}+\frac{1}{2n} }
	\end{align*}
	on $A_n$, since the discrete minimum is not attained at $t=1$, and thus
	\begin{align*}
		A_n\cap
			\left\{  W_\frac{i_n^\ast}{n} - W_{ \frac{i_n^\ast}{n}+\frac{1}{2n} } \geq \frac{c_0}{\eps_0\sqrt{n}} \right\}
			\subseteq \left\{ \abs{ X_1-\Xhat_1 } \geq \frac{c_0}{\sqrt{n}} \right\}.
	\end{align*}
	This yields
	\begin{align*}
		\prob\left(
				\abs{ X_1-\Xhat_1 } \geq \frac{c_0}{\sqrt{n}}
			\right)
			&\geq \prob\left(
				A_n\cap \left\{  W_\frac{i_n^\ast}{n} - W_{ \frac{i_n^\ast}{n}+\frac{1}{2n} } \geq \frac{c_0}{\eps_0\sqrt{n}} \right\}
			\right) \\
		&= \E{ \ind_{A_n}\cdot \prob\left( {   W_\frac{i_n^\ast}{n} - W_{ \frac{i_n^\ast}{n}+\frac{1}{2n} }
			\geq \frac{c_0}{\eps_0\sqrt{n}}  } \cond \A_n \right) }
	\end{align*}
	due to $A_n\in\A_n$. Conditioned on $W_\frac{i}{n}=y_i$ and $W_\frac{i+1}{n}=y_{i+1}$
	for $y_i,y_{i+1}\in\R$, we have
	\begin{align*}
		W_{ \frac{i}{n}+\frac{1}{2n} } \sim \ndist{(y_i+y_{i+1})/2}{1/(4n)}
	\end{align*}
	according to the Brownian bridge construction of $W$. Hence we get
	\begin{align*}
		\ind_{A_n}\cdot  \prob\left( {   W_\frac{i_n^\ast}{n} - W_{ \frac{i_n^\ast}{n}+\frac{1}{2n} }
			\geq \frac{c_0}{\eps_0\sqrt{n}}  } \cond \A_n \right)
			= \ind_{A_n}\cdot f\left( W_\frac{i_n^\ast+1}{n} -W_\frac{i_n^\ast}{n}\right),
	\end{align*}
	where $f:\R\to\R$ is given by
	\begin{align*}
		f(x) = \prob\left(\frac{Z}{\sqrt{4n}}\geq \frac{c_0}{\eps_0\sqrt{n}}+\frac{x}{2}\right)
	\end{align*}
	with $Z\sim\ndist{0}{1}$. Finally, using
	\begin{align*}
		\ind_{A_n}\cdot f\left( W_\frac{i_n^\ast+1}{n}-W_\frac{i_n^\ast}{n} \right)
			\geq \ind_{A_n}\cdot f\left(\frac{1}{\sqrt{n}}\right)
			= \ind_{A_n}\cdot \prob\left( Z\geq \frac{2c_0}{\eps_0}+1 \right),
	\end{align*}
	we obtain
	\begin{align*}
		\prob \left(
				\abs{ X_1-\Xhat_1 } \geq \frac{c_0}{\sqrt{n}}
			\right)
			\geq \prob(A_n)\cdot \prob\left( Z\geq \frac{2c_0}{\eps_0}+1 \right)
			\geq \frac{\varepsilon_0}{2}\cdot \prob\left( Z\geq \frac{2c_0}{\eps_0}+1 \right)
	\end{align*}
	and hence
	\begin{align}\label{eq:case-2}
		\begin{aligned}[c]
			\E{ \abs{ X_1-\Xhat_1 } }
				\geq \frac{c_0}{\sqrt{n}} \cdot \prob\left(\abs{ X_1-\Xhat_1 } \geq \frac{c_0}{\sqrt{n}}\right)
				\geq \frac{c_0\,\eps_0}{2\sqrt{n}} \cdot \prob\left( Z\geq \frac{2c_0}{\eps_0}+1 \right).
		\end{aligned}
	\end{align}
	Combining \eqref{eq:case-1} and \eqref{eq:case-2} completes the proof.
\end{proof}

\section*{Acknowledgement}
We thank James M.~Calvin, Martin Hutzenthaler, and Klaus Ritter
for valuable discussions and comments.

\bibliographystyle{plainnat}
\renewcommand*{\bibname}{References}
\bibliography{references}

\end{document}